\documentclass[12pt]{article}
\usepackage{fullpage}

\usepackage{caption}
\usepackage{url}
\usepackage{color}
\usepackage{amssymb}
\usepackage{amsmath}
\usepackage{amsthm}
\usepackage{graphicx}

\newtheorem{theorem}{Theorem}[section]
\newtheorem{lemma}[theorem]{Lemma}
\newtheorem{fact}[theorem]{Fact}

\newtheorem{cor}[theorem]{Corollary}

\newtheorem{ques}[theorem]{Question}

\title{Partition of Sparse Graphs into Two Forests with Bounded Degree}
\author{Matthew P. Yancey \thanks{Institute for Defense Analyses / Center for Computing Sciences (IDA / CCS), mpyance@super.net}}

\begin{document}
\maketitle

\begin{abstract}
Borodin and Kostochka proved that for $d_2 \geq 2d_1+2$ and a graph $G$ where every subgraph $H$ satisfies 
$$ e(H) < \left(2 - \frac{d_2+2}{(d_1+2)(d_2+1)}\right)n(H) + \frac{1}{d_2+1} $$
has a vertex partition $V(G) = V_1 \cup V_2$ such that $G[V_i]$ has maximum degree at most $d_i$ for each $i$.
We show that under the same conditions we can additionally conclude that each $G[V_i]$ is a forest.
\end{abstract}

\section{Introduction} \label{introduction section}

Let $\mathcal{G}^{(i)}$ denote a family of graphs.
A $(\mathcal{G}^{(1)}, \mathcal{G}^{(2)}, \ldots, \mathcal{G}^{(k)})$-coloring of a graph $G$ is a partition $V(G) = \cup_{i=1}^k V_i$ where $G[V_i] \in \mathcal{G}^{(i)}$ for each $i$.
We primarily focus on three classes of graph families: $\Delta_t$, which is graphs with vertex degrees bounded by $t$; $\mathcal{F}$, which is the family of forests; and $F_i = \mathcal{F} \cap \Delta_i$.
Observe that $\Delta_0 = F_0$ is the set of empty graphs, $\Delta_1 = F_1$ is the set of matchings, and $F_2$ is set of linear forests.
A $(\Delta_{i_1}, \Delta_{i_2}, \ldots, \Delta_{i_k})$-coloring is called a defective coloring.

Appel and Haken \cite{AppelHaken1977,AppelHaken1977_2} proved that every planar graph has a $(\Delta_0, \Delta_0, \Delta_0, \Delta_0)$-coloring in a result known as the Four Color Theorem.
Cowen, Cowen, and Woodall \cite{CowenCowenWoodall1986} proved that every planar graph has a $(\Delta_2, \Delta_2, \Delta_2)$-coloring.
This was improved independently by Poh \cite{Poh1990} and Goddard \cite{Goddard1991} who showed that every planar graph has a $(F_2, F_2, F_2)$-coloring, answering a question of Broere and Mynhardt \cite{BroereMynhardt1984}.
Outerplanar graphs are 2-degenerate, and therefore have a $(\Delta_0, \Delta_0, \Delta_0)$-coloring.
Akiyama, Era, Gervacio, and Watanabe \cite{AkiyamaEraGervacioWatanabe1989} and Broere and Mynhardt \cite{BroereMynhardt1984} showed that outerplanar graphs have a $(F_2, F_2)$-coloring.

Lov\'{a}sz \cite{Lovasz1966} proved that graphs in $\Delta_{-1 + \sum_{i=1}^k (d_i + 1)}$ have $(\Delta_{d_1}, \Delta_{d_2}, \ldots, \Delta_{d_k})$-coloring.
Borodin \cite{Borodin1976} and independently Bollob\'{a}s and Manvel \cite{BollobasManvel1979} showed that every graph in $\Delta_{2t}$ that does not contain $K_{2t+1}$ has a $(F_2, F_2, \ldots, F_2)$-coloring (it has $t$ distinct colors).

We are interested in sparsity conditions that imply a coloring exists.
For a graph $G$, let $e(G) = |E(G)|$ and $n(G) = |V(G)|$.
A graph $G$ is $(a,b)$-sparse if for every subgraph $H \subseteq G$ we have $e(H) \leq a n(H) - b$.
It is $(a,b)$-strictly sparse when $e(H) < a n(H) - b$ for each $H \subseteq G$.
A graph is $(a,b)$-tight if $e(G) = a n(G) - b$, while every proper subgraph of $G$ is $(a,b)$-strictly sparse. 
One of the strongest known results about $(\Delta_j, \Delta_k)$-coloring comes from a sparsity condition.
We say that a graph is $(\mathcal{G}^{(1)}, \mathcal{G}^{(2)})$-critical if it does not have an $(\mathcal{G}^{(1)}, \mathcal{G}^{(2)})$-coloring, but every proper subgraph does.

\begin{theorem}[Borodin and Kostochka \cite{BorodinKostochka2014}]\label{Borodin Kostochka theorem}

\begin{sloppypar}
Let $d_2 \geq 2d_1+2$.
If $G$ is $\left(2 - \frac{d_2+2}{(d_1+2)(d_2+1)}, \frac{-1}{d_2+1}\right)$-strictly sparse graph, then it has a $\left(\Delta_{d_1}, \Delta_{d_2}\right)$-coloring.
Moreover, there exists infinitely many $\left(\Delta_{d_1}, \Delta_{d_2}\right)$-critical graphs that are $\left(2 - \frac{d_2+2}{(d_1+2)(d_2+1)}, \frac{-1}{d_2+1}\right)$-tight.
\end{sloppypar}
\end{theorem}

Our main contribution is to show that the assumptions of Theorem \ref{Borodin Kostochka theorem} also imply a $(F_{d_1}, F_{d_2})$-coloring.
The inspiration for this project is that it is already known to be true in some special cases.
Dross, Montassier, and Pinlou \cite{DrossMontassierPinlou2018} proved that a $(\frac{3}{2}-\frac{1}{2d_2}, 0)$-strictly sparse graph has a $(F_0, F_{d_2})$-coloring when $d_2 \geq 3$.
Chen, Yu, and Wang \cite{ChenYuWang2018} improved this to $(\frac{3}{2} - \frac{1}{2(d_2+1)}, \frac{-1}{d_2+1})$-strictly sparse graphs having a $(F_0, F_{d_2})$-coloring when $d_2 \geq 2$, which matches Borodin and Kostochka's formula when $d_1=0$.
Chen, Raspaud, and Yu \cite{ChenRaspaudYu2022} proved that a $(1.6,0)$-sparse graph has a $(F_1, F_4)$-coloring (observe that $(\frac{8}{5}, 0)$-sparse is equivalent to Borodin and Kostochka's formula of $(\frac{8}{5}, \frac{-1}{5})$-strictly sparse for $(\Delta_1, \Delta_4)$-coloring).

\begin{theorem}\label{main theorem}
Let $d_2 \geq 2 d_1 + 2$.
If $G$ is a $(F_{d_1}, F_{d_2})$-critical multigraph without loops, then 
$$e(G) \geq \left(2 - \frac{d_2 + 2}{(d_1 + 2)(d_2 + 1)}\right)n(G) + \frac{1}{d_2 + 1}.$$
Moreover, if $G$ is a $\left(2 - \frac{d_2+2}{(d_1+2)(d_2+1)}, \frac{-1}{d_2+1}\right)$-strictly sparse multigraph without loops, then a $(F_{d_1}, F_{d_2})$-coloring of $G$ can be found in polynomial time.
\end{theorem}

Let $g(G)$ denote the girth of $G$.
It is well-known that planar graphs are $(\frac{g(G)}{g(G)-2},0)$-strictly sparse.
Let $\widehat{G}_\ell$ denote planar graphs with girth $\ell$ with no adjacent $\ell$-cycles.
A trivial discharging argument reveals that graphs in $\widehat{G}_\ell$ are $(\frac{\ell(\ell+1)}{\ell^2-\ell-1},0)$-strictly sparse.
Huang, Huang, and Lv \cite{HuangHuangLv2023} showed that planar graphs without cycles of length $4$ to $i$ are $(\frac{3i-3}{2(i-2)},0)$-strictly sparse.
These statements lead to quick corollaries in topological graph theory from results connecting sparsity to colorings.

For example, Dross, Montassier, and Pinlou \cite{DrossMontassierPinlou2018} showed that planar graph $G$ has a $(F_0, F_5)$-coloring if $g(G) \geq 7$, a $(F_0, F_3)$-coloring if $g(G) \geq 8$, and $(F_0, F_2)$-coloring if $g(G) \geq 10$.
Chen, Yu, and Wang \cite{ChenYuWang2018} improved this to a $(F_0, F_4)$-coloring if $g(G) \geq 7$ and a $(F_0, F_2)$-coloring if $g(G) \geq 8$.
Chen, Raspaud, and Yu \cite{ChenRaspaudYu2022} showed that if $G$ has genus at most $1$ and girth at least $6$, then it has a $(F_1, F_4)$-coloring.
It also follows quickly that if $G \in \widehat{G}_5$, then $G$ has a $(F_1, F_4)$-coloring.
Theorem \ref{main theorem} implies that such colorings can be found in polynomial time.
We also establish similar corollaries.

\begin{cor}\label{implications of main theorem}
Let $G$ be a planar graph.
\begin{itemize}
	\item[(A)] If $g(G) \geq 5$, then $G$ has a $(F_2, F_6)$-coloring.
	\item[(B)] If $G \in \widehat{G}_4$, then $G$ has a $(F_5, F_{12})$-coloring.
	\item[(C)] If $G$ has no cycles of length $4$ to $11$, then $G$ has a $(F_2, F_{6})$-coloring.
\end{itemize}
Moreover, the above colorings can be found in polynomial time.
\end{cor}

Choi and Raspaud \cite{ChoiRaspaud2015} asked if planar graphs with girth $5$ have $(\Delta_1, \Delta_7)$-colorings.
As partial progress, Zhang, Chen, and Wang \cite{ZhangChenWang2017} proved that graphs in $\widehat{G}_5$ have $(\Delta_1, \Delta_7)$-colorings.
This result is improved by the above discussion.
Wang, Huang, and Finbow \cite{WangHuangFinbow2020} show that graphs in $\widehat{G}_5$ have $(F_3, F_3)$-colorings.

Dross, Montassier, and Pinlou \cite{DrossMontassierPinlou2015} showed that a planar graph with girth $4$ has a $(F_5, \mathcal{F})$-coloring.
This was improved by Feghali and \v{S}\'amal \cite{FeghaliSamal2024} to a $(F_3, \mathcal{F})$-coloring.
Liu and Wang \cite{LiuWang2022} showed that a planar graph with girth $4$ and no chorded $6$-cycle (which includes all graphs in $\widehat{G}_4$) has a $(F_2, \mathcal{F})$-coloring.
On the other hand, Montassier and Ochem \cite{MontassierOchem2013} constructed for any finite $d_1, d_2$ a planar graph with girth $4$ that has no $(\Delta_{d_1}, \Delta_{d_2})$-coloring.

Cho, Choi, and Park \cite{ChoChoiPark2021} showed that a planar graph without $4$- or $5$-cylces has a $(F_3, F_4)$-coloring, while asking if such a graph has a $(F_2, F_{d_2})$-coloring for some finite $d_2$.


There are three technical comments to be made about Theorem \ref{main theorem}.
First, the theorem considers multigraphs (we consider parallel edges to be cycles of length $2$, and thus any parallel edge must have endpoints with distinct colors in a $(F_{d_1}, F_{d_2})$-coloring).
In similar situations (see \cite{CranstonYancey2020} and \cite{DPcoloring0,DPcoloring} and \cite{digraphs}) the optimal sparsity conditions for a coloring depends on whether it is restricted to simple graphs, although that did not happen here.
Second, the polynomial-time algorithm is in contrast to results about similar colorings being NP-complete.
For example, Montassier and Ochem \cite{MontassierOchem2013} showed that it is NP-complete to determine whether a planar graph with girth $5$ has a $(\Delta_1, \Delta_3)$-coloring.
Moreover, polynomial-time algorithms have generated independent interest (for example, see \cite{cocoa} for an interest in algorithms to produce a $(\Delta_1, \Delta_1)$-coloring or Setion 5.1 of \cite{CorsiniDeschampsFeghaliGoncalvesLangloisTalon2023} for $(F_2, \ldots, F_2)$-coloring).
Third, let us return to the fact that in Theorems \ref{Borodin Kostochka theorem} and \ref{main theorem} that the optimal sparsity condition for a $(\Delta_j, \Delta_k)$-coloring is the same as for a $(\Delta_j \cap \mathcal{F}, \Delta_k \cap \mathcal{F})$-coloring.
We informally describe this phenominon as ``acyclic is free in the sparsity context'' for defective coloring.
Cranston and the author \cite{CranstonYancey2021} encountered this phenominon previously for a different type of coloring.
Which leads us to ask the informal question, for what other types of colorings is acyclic free in the sparsity context?
This question is uninteresting for defective colorings, as the remaining cases where optimal sparsity conditions are known are $(\Delta_i, \Delta_j)$-coloring where $i,j \leq 1$ (see \cite{BorodinKostochka2011} and \cite{BorodinKostochkaYancey}), and in such cases $\Delta_i = F_i$, $\Delta_j = F_j$.
But the question is interesting for other types of colorings (for example, defective DP-coloring \cite{DPcoloring0,DPcoloring}).

The paper is organized as follows.
In Section \ref{preliminaries section} we cover technical preliminaries and establish notation.
Sections \ref{gap lemma section}, \ref{reducible configurations section}, and \ref{discharging section} are dedicated to proving the first half of Theorem \ref{main theorem}.
We use the potential method, which is a discharging proof enhanced with an extra tool called the ``gap lemma.''
The gap lemma is established in Section \ref{gap lemma section}.
The two classical phases of a discharging proof, the reducible configurations and the discharging rules, are presented in Sections \ref{reducible configurations section} and \ref{discharging section}, respectively.
In Section \ref{algorithm section} we prove the second half of Theorem \ref{main theorem}---the polynomial algorithm---by reviewing the arguments made in Sections \ref{gap lemma section} and \ref{reducible configurations section}.
Section \ref{conjecture section} ends this paper with a discussion on whether the assumption $d_2 \geq 2d_1 + 2$ is sharp.

\section{Technical Preliminaries} \label{preliminaries section}

\subsection{Notation} \label{notation subsection}

We use $N^m(u)$ to denote the multiset $\{x: e \in E(G), e=ux\}$ such that $d(u) = |N^m(u)|$ (`m' stands for multiplicity).
Similarly we define the closed multi-neighborhood as $N^m[u] := \{u\} \cup N^m(u)$.
For vertex set $S$, let $N^m(u) \cap S$ denote the multisubset of $N^m(u)$ formed from restricting it to elements in $S$.
Let $N(u)$ denote the underlying set of $N^m(u)$.
For subgraph $H$ and vertex $x$, let $H + x$ denote the subgraph induced on $V(H) \cup \{x\}$ and $H - x$ denote the subgraph induced on $V(H) \setminus \{x\}$.

For $i \in \{1,2\}$, let $W_i: V \rightarrow \mathbb{Z}_{\geq 0}$ denote two weight functions.
A triple $(G,W_1,W_2)$ is denoted as a \emph{weighted graph}.
We say that a partition $V = V_1 \cup V_2$ is a desired $(d_1, d_2)$-coloring if each $G[V_j]$ is a forest and for each $x \in V_j$ we have that $|N^m(x) \cap V_j| + W_j(x) \leq d_j$.
If the values of $d_1, d_2$ are clear, then we simply call it a \emph{desired coloring}.
We say that $(H,W_1',W_2')$ is a \emph{weighted subgraph} of $(G,W_1,W_2)$ if $H$ is a subgraph of $G$ and $W_j'(u) \leq W_j(u)$ for all $j,u$.
Moreover, it is an \emph{induced} weighted subgraph if $H$ is induced and $W_i' = W_i$ over $V(H)$.
It should be clear that if $(G, W_1, W_2)$ has a desired coloring, then so does every weighted subgraph.
A graph is called \emph{critical} if there is no desired coloring of it, but each proper weighted subgraph can be desirably colored.

This notation is similar to prior work.
Chen, Yu, and Wang \cite{ChenYuWang2018} use a function $s$ that is equivalent to our $W_2$.
Our notation $(G, W_1, W_2)$ is implicitly equivalent to the notation by Chen, Raspaud, and Yu \cite{ChenRaspaudYu2022} of $(G^* - P, ph, f)$, where $P$ is the set of pendant hosts.
The following notation of capacities is the same after translation by $1$ as used by Kostochka, Xu, and Zhu \cite{KostochkaXuZhu2023}.

Define $c_j(u) = d_j + 1 - W_j(u)$, which we will call the \emph{capacity} of $u$ for color $j$.
The inequality for a desired coloring can be rewritten as $u \in V_j$ implies $|N^m[u] \cap V_j| \leq c_j(u)$.
If $u$ has zero (or less) capacity for color $j$, then a desired coloring has $u \in V_{3-j}$.
Capacity and weight are redundant terminology, but having both will make reading our arguments easier.
Weight is useful for understanding concepts like ``subgraph;'' while capacity provides a cleaner reference for when a desirable coloring does or does not exist.
In particular, we will frequently discuss vertices with zero capacity and therefore their coloring is predetermined.

We use
$$\alpha = \frac{d_2 + 2}{(d_1 + 2)(d_2 + 1)} \hspace{0.1in} \mbox{ and } \hspace{0.1in} \beta = \frac{1}{d_2 + 1}. $$
The theorem that we will prove in Sections 3, 4, and 5 is as follows.
If $W_1(v) = W_2(v) = 0$, then we say $v$ is \emph{weightless}.
The first half of Theorem \ref{main theorem} follows from Theorem \ref{central goal} by making every vertex weightless.

\begin{theorem}\label{central goal}
Let $d_1, d_2, W_1, W_2$ be fixed with $d_2 \geq 2 d_1 + 2$.
If $G$ is critical, then 
$$ e(H) \geq \beta + \sum_{x \in V(H)}\left( \alpha c_1(x)  + \beta(c_2(x)-1) \right) .$$
\end{theorem}

For a weighted graph $(H, W_1, W_2)$, we denote the \emph{potential} of $H$ as 
$$ \rho_{d_1, d_2}(H, W_1, W_2) = \sum_{x \in V}\left( \alpha c_1(x)  +  \beta (c_2(x)-1) \right) - e(H). $$
When $d_1, d_2, W_1, W_2$ are clear from context, we will simply write $\rho(H)$ for the potential.
Furthermore, if $H$ is an induced weighted subgraph of $G$ and $G$ is clear from context, then we may simply write $\rho(V(H))$.
Observe that if $(H, W_1', W_2')$ is a weighted subgraph of $(G, W_1, W_2)$, then $\rho(H, W_1', W_2') \geq \rho(V(H))$ as decreasing the weight and removing edges only increases the potential.
Theorem \ref{central goal} is equivalent to $\rho(G) \leq -\beta$ for any critical $G$.

\subsection{Arithmetic and Constructions} \label{arithmetic and constructions section}

In the following we summarize arithmetic that will be repeated multiple times.
For the reader's benefit we will reference Fact \ref{basic arithmetic with potential} at appropriate times, although we use Fact \ref{basic arithmetic with potential}.2.e too frequently to refer to it each time.

\begin{fact}\label{basic arithmetic with potential}
In the following, we assume that the capacity is nonnegative.
\begin{enumerate}
	\item As $d_1 \geq 0$ and $d_2 \geq 2$ we have that $\alpha \leq 2/3$ and $\beta \leq 1/3$ and so $\alpha + \beta \leq 1$.
	\item The potential of a single vertex $u$ for different capacities:
	\begin{enumerate}
		\item If $c_1(u) = 1$ and $c_2(u) = 0$, then $\rho(\{u\}) = \alpha - \beta > 0$.
		\item If $c_1(u) = 0$ and $c_2(u) = 1$, then $\rho(\{u\}) = 0$. 
		\item If $c_1(u) = 0$, then $\rho(\{u\}) \leq 1 - \beta$.
		\item If $c_2(u) = 0$, then $\rho(\{u\}) \leq 1 - \alpha$.
		\item For every vertex $u$, $\rho(\{u\}) \leq 2 - \alpha$.
	\end{enumerate}
	\item If $c_1(u) = c_1(v)$ and $c_2(u) \geq c_2(v)$, then $\rho(\{u\}) - \rho(\{v\}) \leq 1$. 
	\item If $c_2(u) = c_2(v)$ and $c_1(u) \geq c_1(v)$, then $\rho(\{u\}) - \rho(\{v\}) \leq 1 - \alpha + \beta$. 
	\item Adding an unweighted vertex and two edges will change the potential an equal amount as incrementing $W_1$ by one for a single vertex.
	\item As $\alpha = \beta \frac{d_2 + 2}{d_1 + 2}$, we have that $\alpha \geq 2\beta$ if and only if $d_2 \geq 2 d_1 + 2$.  
	\item When $d_2 \geq 2 d_1 + 2$ we have that increasing $W_2$ by one for two vertices changes the potential at most as adding an unweighted vertex and two edges.
\end{enumerate}
\end{fact}

The simplest critical graph is an isolated vertex with capacity $0$ in both colors, which is sharp for Theorem \ref{central goal}.
The sharp examples described by Borodin and Kostochka \cite{BorodinKostochka2014} in Theorem \ref{Borodin Kostochka theorem} are also critical weightless graphs, which proves the sharpness of Theorem \ref{central goal}.
We will describe those examples in the language of Section \ref{notation subsection}, which will provide intuition for some arguments we make later.
Borodin and Kostochka's examples are created by a construction with several steps, which we will break into two constructions.

For the first construction, let $(G,W_1,W_2)$ be critical and sharp for Theorem \ref{central goal} and contain a vertex $v$ with $W_1(v) > 0$.
A new sharp example can be constructed from $G$ by decrementing $W_1(v)$ by one and appending a leaf $v'$ to $v$ such that $W_1(v') = c_2(v') = 0$.
Borodin and Kostochka \cite{BorodinKostochka2014} use the phrase ``peripheral $(d_2, d_1)$-host'' for this construction, while Chen, Raspaud, and Yu \cite{ChenRaspaudYu2022} call it a ``pendant host.''

The second construction requires the following set up:
\begin{itemize}
	\item Let $\{(G_i, W_{1,i}, W_{2,i})\}_{1 \leq i \leq \ell}$ with $\ell \leq d_1 + 2$ be a nonempty set of critical graphs that are sharp for Theorem \ref{central goal}.
	\item Suppose each $(G_i, W_{1,i}, W_{2,i})$ contains a vertex $v_i$ with $W_{2}(v_i) > 0$.
	\item Let $H$ be a $K_{1,d_1+1}$ with each vertex unweighted and vertex set $u_1, \ldots, u_{d_1+2}$.
	\item Let $y_1,y_2,\ldots,y_{d_1+2}$ be a sequence of vertices, each chosen from $\{v_1, \ldots, v_\ell\}$, such that each $v_i$ is chosen at least once.
\end{itemize}
Under these circumstances, a new sharp critical graph can be constructed from $H \cup \bigcup _{1 \leq i \leq \ell} G_i$ by adding edge $y_j u_j$ for each $j$ and decrementing $W_{2}(v_i)$ by one for each $i$ (regardless of how many neighbors in $H$ it has).
Borodin and Kostochka \cite{BorodinKostochka2014} use the term ``flag'' for $H$ in the special case of this construction when $\ell = 1$.
Chen, Yu, and Wang \cite{ChenYuWang2018} used the name ``special triangles'' for flags when $d_1 = 0$.

Any critical graph can be expanded into a weightless critical graph by repeatedly applying the first construction and then repeatedly adding flags.
Moreover, if the original graph is sharp for Theorem \ref{central goal}, then the resulting weightless graph will be too.

There is a third construction that we can introduce.
Let $(G,W_1,W_2)$ be critical and sharp for Theorem \ref{central goal}.
If $c_i(u) = 0$ for some vertex $w$ and index $i$, then form a new sharp critical graph by reducing $W_i(u)$ to zero, appending a leaf $u'$ to $u$, and giving $u'$ capacities $c_i(u') = 1$, $c_{3-i}(u') = 0$.
The proof to our gap lemma will be slightly different than in \cite{BorodinKostochka2014,ChenRaspaudYu2022}, and this is because we will use this third construction to shorten our argument.

\subsection{Opening Statements of the Proof}\label{opening statements section}
We assume $d_1, d_2$ are fixed integers such that $d_2 \geq 2d_1 + 2$.
We say that weighted graph $(H', W_1', W_2')$ is \emph{smaller} than weighted graph $(H, W_1, W_2)$ if it contains fewer edges.
By way of contradiction, let $G$ be a smallest counterexample to Theorem \ref{central goal}.
That is, $G$ is a critical graph with $\rho(G) > -\beta$ and any critical graph $H$ with fewer edges than $G$ satisfies $\rho(H) \leq -\beta$.

We make a couple of obvious statements about $G$, as they are true for any critical graph.
We can assume $G$ is connected.
We can assume $c_j(w) \geq 0$ for all $w \in V$ and $j$.
It should also be clear that $c_1(w) + c_2(w) \geq 1$ for all $w \in V$, as this is true for any critical graph with an edge.
By simple case analysis, we can assume $G$ contains at least two edges.
The multiplicity of any edge in $G$ is at most $2$ (if it is at least two, then its endpoints have different colors in any desirable coloring).

\section{Gap Lemma} \label{gap lemma section}

\begin{lemma}\label{gap lemma}
If $(H, W_1', W_2')$ is a nonempty weighted subgraph of $(G, W_1, W_2)$ with $e(H) < e(G)$, then $\rho(H, W_1', W_2') > \alpha - \beta$.
\end{lemma}
\begin{proof}
Let $(H, W_1', W_2')$ be a proper weighted subgraph of $(G, W_1, W_2)$ that minimizes $\rho(H, W_1', W_2')$ among all those with at least one fewer edge than $G$.
By way of contradiction, assume $\rho(H) \leq \alpha - \beta$.
Reducing weight only increases potential, so we may assume that $W_1' = W_1$ and $W_2' = W_2$.

If $V(H) = V(G)$, then $\rho(H) \geq \rho(G) + 1 > -\beta + 1 > \alpha - \beta$.
Therefore $V(H) \subsetneq V(G)$.
As removing edges only increases potential, we may assume that $H$ is an induced weighted subgraph.

Our argument now splits into two cases, although the arguments are similar in each.
It will benefit the reader to observe now that the assumption of case 1 is not used until the last line of the argument.

\textbf{Case 1:} $\rho(H) \leq 0$.
As $G$ is critical, there is a desired coloring of $H$; let $V(H) = V_1 \cup V_2$ be such a coloring.
We construct a new weighted graph $(G^*, W_1^*, W_2^*)$ as follows:
\begin{itemize}
	\item[(A)] Let $G^* = G - V(H)$.
	\item[(B)] Let $W_i^* = W_i$, with the following exceptions:
	\item[(C)] If $u \notin V(H)$ and $N_G(u) \cap V_i \neq \emptyset$, then set $c_i^*(u) = 0$.
\end{itemize}
If there is a desired coloring of $(G^*, W_1^*, W_2^*)$, then unioning that coloring with the coloring of $H$ would create a desired coloring of $G$, as by construction there is no monochromatic edge that bridges $V(H)$ and $G \setminus V(H)$.
This contradicts the criticality of $G$, so there must be a critical subgraph $(\widehat{G}, \widehat{W_1}, \widehat{W_2})$ of $(G^*, W_1^*, W_2^*)$.
As it has fewer edges, $\widehat{G}$ is smaller than $G$, and so by assumption $\rho(\widehat{G}, \widehat{W_1}, \widehat{W_2}) \leq -\beta$.

Let $(\widehat{H}, W_1, W_2)$ be the weighted induced subgraph of $(G, W_1, W_2)$ over the vertex set $V(H) \cup V(\widehat{G})$.
As the potential function is linear, we can compute $\rho(\widehat{H}, W_1, W_2)$ as a sum of `contributions' from four steps:
\begin{enumerate}
	\item[(1)] the sum of $\rho(H, W_1, W_2)$ and $\rho(\widehat{G}, \widehat{W_1}, \widehat{W_2})$,
	\item[(2)] plus $\rho(V(\widehat{G}), W_1^*, W_2^*) - \rho(\widehat{G}, \widehat{W_1}, \widehat{W_2})$,
	\item[(3)] plus $\rho(V(\widehat{G}), W_1, W_2) - \rho(V(\widehat{G}), W_1^*, W_2^*)$, which comes from the change in capacity in step (C) in the construction of $G^*$, and
	\item[(4)] minus $1$ for each edge with an endpoint in $H$ and another endpoint in $\widehat{G}$. 
\end{enumerate}
Observe that as $\widehat{G}$ is a weighted subgraph of $G^*$, the contribution of step (2) is nonpositive.
By Fact \ref{basic arithmetic with potential}.(3-4), each application of step (C) in the construction of $G^*$ affects the potential by at most $1$.
By construction, there is an injection from each application of (C) on a vertex in $\rho(\widehat{G})$ to the edges with an endpoint in $H$ and another endpoint in $\widehat{G}$.
Thus the net overall contribution of steps (3) and (4) above is nonpositive.
Therefore the overall potential is bounded by step (1) above and we have 
$$\rho(\widehat{H}, W_1, W_2) \leq \rho(H, W_1, W_2) + \rho(\widehat{G}, \widehat{W_1}, \widehat{W_2}) \leq \rho(H, W_1, W_2)-\beta.$$

So the choice of $H$ as minimizing potential among all proper subgraphs implies that $\widehat{H} = G$.
But as $\rho(H) \leq 0$, this implies $\rho(G) \leq -\beta$, which contradicts the choice of $G$.

\textbf{Case 2:} $\rho(H) > 0$.  
This case follows similarly to case 1.
The difference is that we increase the weight on a vertex in $H$ to influence the desired coloring of $H$, which is then used to slightly modify the construction of $G^*$.
In the following we carefully describe the new aspects of the argument, while quickly mentioning details that are unchanged from case 1.

Pick an arbitrary edge $zx$ such that $z \in H$ and $x \notin H$.
Let $(H', W_1', W_2')$ be the weighted graph formed from $H$ by increasing $W_2(z)$ by one.
By choice of $H$, the assumption of this case implies that $\rho(H^*) > 0$ for all proper weighted subgraphs $H^*$ of $G$.
Therefore every weighted subgraph of $H'$ has potential strictly larger than $-\beta$.
By the minimality of $G$, this implies that $H'$ does not contain a critical subgraph, and therefore $H'$ has a desirable coloring.
Let $V(H') = V_1' \cup V_2'$ be such a coloring.

\textbf{Case 2.A:} $z \in V_1'$.
This subcase proceeds exactly like case 1, except with a stronger analysis of the potential function at the end.
Observe that the entire argument except the last line of case 1 still holds (including the conclusion that $\widehat{H} = G$), but  
$$ \rho(G) = \rho(\widehat{H}, W_1, W_2) \leq \rho(H) - \beta $$
is no longer a contradiction as we only get $\rho(G) \leq \alpha - 2\beta$.
The key difference is that during the construction of $G^*$, when we set $c_1(x) = 0$ during step (C), we only modified the potential by at most $1- \alpha+\beta$ by Fact \ref{basic arithmetic with potential}.4, when we previously bounded it by $1$.
As we have $\widehat{H} = G$, we know $x \in \widehat{G}$, and therefore we can strengthen our bound on the potential by $\alpha-\beta$, recreating the contradiction as before.

\textbf{Case 2.B:} $z \in V_2'$.
As $W_2'(z) = W_2(z) + 1$, we have that $W_2(z) + |N^m(z) \cap V_2'| \leq d_2 - 1$.
We now have a slightly different approach to constructing $G^*$.

First, observe that $N^m(x) \cap V(H) = \{z\}$, because if $x$ is in a second edge with an endpoint in $H$ then 
$$\rho(H + x) \leq \rho(H) + (2-\alpha) - 2 \leq -\beta.$$
This either contradicts the minimality of $\rho(H)$ (if $H + x$ is a proper subgraph of $G$) or the choice $\rho(G) > -\beta$.
Therefore step (C) in the construction of $G^*$ is only applied at most once to $x$, and it is because of the edge $zx$.
So we construct $G^*$ the same way as in case 1, with the following exception: when we apply step (C) to $x$, instead of setting $c_2^*(x) = 0$, we instead set $W^*_2(x) = W_2(x) + 1$.
  
We repeat the claim that if $G^*$ has a desirable coloring $V_1^* \cup V_2^*$, then unioning it with desirable coloring $V_1' \cup V_2'$ of $H$ would lead to a desirable coloring of $G$, which is a contradiction.
By construction, the only edge with exactly one endpoint in $H$ that could have both endpoints in the same color class is $zx$.
Since there is only one such edge, it can not be in a monochromatic cycle.
Because $z \in V_2'$, if $xz$ is monochromatic, then $x \in V_2^*$, and by construction $W_2(z) + |N^m(z) \cap V_2^*| \leq d_2 - 1$.
This proves the claim.

So $G^*$ has a critical subgraph $\widehat{G}$, and we repeat the argument from case 1 until the last line, where similar to case 2.A we use a stronger analysis of the potential function.
This time we bound by $\beta$ the affect of step (C) on $x$, and therefore we have $\rho(G) \leq (\alpha-\beta) - \beta - (1-\beta) < - \beta$, recreating the contradiction from before.
\end{proof}

\begin{cor}\label{at least -beta}
If $H$ is a weighted subgraph of $G$, then $\rho(H) > -\beta$.
\end{cor}
\begin{proof}
If $H$ is empty, then $\rho(H) = 0$.
If $V(H) = V(G)$, then $\rho(H) \geq \rho(G) > -\beta$.
Otherwise $\rho(H) > \alpha - \beta$ by Lemma \ref{gap lemma}.
\end{proof}

\begin{cor}\label{capacity sums to at least 2}
For each vertex $u$ in $G$ we have $c_1(u) + c_2(u) \geq 2$.
\end{cor}

The most common application of Lemma \ref{gap lemma} is to say that proper subgraphs of $G$ have desirable colorings, even if we increase the weights on a few vertices.
We collect some of these arguments below.

\begin{cor} \label{add a flag or two}
Let $(H, W_1', W_2')$ be a proper weighted subgraph of $(G, W_1, W_2)$ with strictly fewer edges.
There exists a desired coloring of $(H, W_1'', W_2'')$ in any of the following two cases:
\begin{itemize}
	\item[(I)] $W_1' = W_1''$, $W_2' = W_2''$ except for vertex $u$ and index $i$, where $W_i''(u) = W_i'(u) + 1$.
	\item[(II)] $W_1' = W_1''$, $W_2' = W_2''$ except there is a vertex $x \notin V(H)$ and vertex set $\{u_1, u_2, \ldots, u_k\} \subseteq N(x) \cap V(H)$ where $W_{i_j}''(u_j) = W_{i_j}'(u_j) + 1$ for some sequence $i_j$ and at least one of the following hold:
	\begin{itemize}
		\item[(A)] $e(G) \geq e(H) + k + 1$, 
		\item[(B)] $i_j = 2$ for all $j$, 
		\item[(C)] $\rho(\{x\}) \leq 2-2\alpha$, or
		\item[(D)] $k \geq 3$ and $d_1 \geq 1$.

	\end{itemize}
\end{itemize}
\end{cor}

\begin{proof}
By way of contradiction, suppose that $(H, W_1'', W_2'')$ has a critical weighted subgraph $(\widehat{H}, \widehat{W_1}, \widehat{W_2})$.
By the minimality of $G$ we have that $\rho(\widehat{H}) \leq -\beta$.
By Lemma \ref{gap lemma}, $\rho(V(\widehat{H}), W_1, W_2) > \alpha - \beta$.
This combines to say that 
$$ \rho(V(\widehat{H}), W_1, W_2) - \rho(\widehat{H}, \widehat{W_1}, \widehat{W_2}) > (\alpha - \beta) - (-\beta) = \alpha. $$
Increasing one of the weights of a single vertex only changes the potential by at most $\alpha$, and so this proves part (I).

Let $k' = |V(\widehat{H}) \cap \{x_1, \ldots, x_k\}|$.
By the above reasoning, we can assume $k' \geq 2$.

Observe that 
\begin{eqnarray*}
 \rho(\widehat{H} + x, W_1, W_2)  &\leq & \rho(V(\widehat{H}), W_1, W_2) + \rho(\{x\}) - k' \\
			& \leq & (\rho(V(\widehat{H}), \widehat{W_1}, \widehat{W_2}) + k' \alpha) + (2 - \alpha) - k' \\
			& \leq & \alpha-\beta - (k'-2)(1-\alpha).
\end{eqnarray*}
As $\alpha < 1$, so $\rho(\widehat{H} + x) \leq \alpha - \beta$.
By Lemma \ref{gap lemma}, $\widehat{H} + x = G$, which implies $k' = k$.
We strengthen this inequality based on additional assumptions.

By the formula for $\alpha$ we have that $\alpha \leq 2/3$ and if $d_1 \geq 1$, then $\alpha < 0.5$.
If we are in Case (II.A), then we missed an edge and the above bound becomes $\rho(\widehat{H} + x) \leq \alpha - \beta - 1 < -\beta$, which contradicts Corollary \ref{at least -beta}.
If we are in Case (II.B), then the second line in the above inequality $k' \alpha$ can be replaced by $k' \beta$, and we get $\rho(\widehat{H} + x) \leq \beta-\alpha$.
By Fact \ref{basic arithmetic with potential}.6 we have $\beta - \alpha \leq -\beta$, which contradicts Corollary \ref{at least -beta}.
If we are in Case (II.C), then the bound $\rho(\{u\}) \leq 2-\alpha$ can be replaced by $2-2\alpha$, and so $\rho(\widehat{H} + x) \leq -\beta$, which contradicts Corollary \ref{at least -beta}.
If we are in Case (II.D), then $\rho(\widehat{H} + x) \leq \alpha - \beta - (1-\alpha) < -\beta$, which contradicts Corollary \ref{at least -beta}.
\end{proof}

\section{Reducible Configurations} \label{reducible configurations section}

We will use the following notation about a vertex $u$:
\begin{itemize}
	\item $u$ is $i$-slack if $c_i(u) \geq d(u) + 1$ or if $c_i(u) = d(u)$ and $c_{3-i}(u) \geq 1$;
	\item $u$ is $i$-null if $c_i(u) = 0$; and 
	\item $u$ is $i$-constrained if it is not $i$-slack or $i$-null.
\end{itemize}
We say that a vertex is \emph{doubly-constrained} if it is $1$-constrained and $2$-constrained.
Observe that a vertex $u$ is doubly-constrained if and only if $1 \leq c_i(u) < d(u)$ for each $i$.
We say that a vertex is \emph{somehow-constrained} if it is $1$-constrained or $2$-constrained.
We will use the term \emph{three-two-two} to denote a vertex $u$ with degree $3$ and $c_i(u) \geq 2$ for each $i$.
We will use the term \emph{triple-three} to denote a vertex $u$ with degree $3$ and $c_i(u) \geq 3$ for each $i$.

We will give a quick outline of our reducible configurations that is more intuitive than rigorous (for the extent of this outline, ignore that a vertex can be both doubly-constrained and a three-two-two).
In Lemmas \ref{2 neighbors} and \ref{degree 2} we will show that $G$ has minimum degree $2$, and vertices with degree $2$ are doubly-constrained or $i$-null for some $i$.
This is sufficient to prove the theorem if $d_1 = 0$.
When $d_1 > 0$, any vertex with insufficient charge will be a three-two-two.
We balance the charge of such vertices as follows.
In Lemma \ref{triple-three next to a double constrained} we will show that each three-two-two vertex has a doubly-constrained neighbor when $d_1 \geq 1$.
In Lemma \ref{triple three next to two somehow-constrained} we will show that if the vertex satisfies the stronger property of being triple-three, then it is adjacent to another vertex that is somehow-constrained.

\begin{lemma}\label{2 neighbors}
For each vertex $u$, $|N(u)| \geq 2$.
\end{lemma}
\begin{proof}
By way of contradiction, suppose there exists a vertex $u$ such that $N^m(u) = \{x\}$ or $N^m(u) = \{x, x\}$.
Let $H = G - u$.
By criticality of $G$, $H$ has a desirable coloring.
If $c_1(u), c_2(u) \geq 1$, then the coloring can be extended to a desirable coloring of $G$ by giving $u$ the opposite color of $x$.
This is a contradiction, so there exists an $i$ such that $u$ is $i$-null.

Suppose first that $ux$ is a simple edge.
By Corollary \ref{capacity sums to at least 2} we have $c_{3-i}(u) \geq 2$.
By Corollary \ref{add a flag or two}.I, there is a desirable coloring of weighted graph $(H, W_1', W_2')$ where $W_j' = W_j$ except $W_{3-i}'(x) = W_{3-i}(x) + 1$.
By construction, such a coloring extends to a desirable coloring of $G$, which is a contradiction.

So $ux$ is a parallel edge.
Let $(H, W_1'', W_2'')$ be the weighted graph where $W_j'' = W_j$, except $c_{3-i}''(x) = 0$.
If $i=1$, then by Fact \ref{basic arithmetic with potential}.2.c and Fact \ref{basic arithmetic with potential}.3 we have $\rho(\{u\}) \leq 1-\beta$ and $\rho(\{x\}, W_1, W_2) - \rho(\{x\}, W_1'', W_2'') \leq 1$.
If $i=2$, then by Fact \ref{basic arithmetic with potential}.2.d and Fact \ref{basic arithmetic with potential}.4 we have $\rho(\{u\}) \leq 1-\alpha$ and $\rho(\{x\}, W_1, W_2) - \rho(\{x\}, W_1'', W_2'') \leq 1-\alpha+\beta$.
So in either case we have $\rho(\{u\}) + \rho(H, W_1, W_2) - \rho(H, W_1'', W_2'') \leq 2-\beta$.

We claim that $(H, W_1'', W_2'')$ has a desirable coloring.
By way of contradiction, suppose $(H, W_1'', W_2'')$ contains a critical weighted subgraph $(\widehat{H}, \widehat{W_1}, \widehat{W_2})$.
As $G$ does not contain a proper weighted subgraph that is critical, we have $c \in V(\widehat{H})$.
By the minimality of $G$, we have $\rho(\widehat{H}, \widehat{W_1}, \widehat{W_2}) \leq -\beta$.
As it is a subgraph, we have $\rho(G[V(\widehat{H})], W_1'', W_2'') \leq \rho(\widehat{H}, \widehat{W_1}, \widehat{W_2})$.
So we have
\begin{eqnarray*}
 \rho(\widehat{H} + u, W_1, W_2) & \leq & \rho(V(\widehat{H}), W_1'', W_2'') + \rho(\{u\}) + \rho(V(\widehat{H}), W_1, W_2) - \rho(V(\widehat{H}), W_1'', W_2'') - 2 \\
				& \leq & -\beta + (2-\beta) -2 < -\beta.\\
\end{eqnarray*}
This contradicts Corollary \ref{at least -beta}, which proves the claim.
The desirable coloring of $(H, W_1'', W_2'')$ clearly extends to a desirable coloring of $(G, W_1, W_2)$, contradicting the choice of $G$.
\end{proof}

\begin{lemma}\label{degree 2}
If $d(u) = 2$, then $u$ is doubly-constrained or there exists an $i$ such that $u$ is $i$-null.
\end{lemma}
\begin{proof}
Let $N(u) = \{x, y\}$.
By Lemma \ref{2 neighbors} we have $x \neq y$.
We have that $c_1(u) = c_2(u) = 1$ if and only if $u$ is doubly-constrained.
By way of contradiction, assume that $\max(c_1(u), c_2(u)) \geq 2$ and $\min(c_1(u), c_2(u)) \geq 1$.
Let $H = G - u$.

\textbf{Case 1:} $c_2(u) \geq 2$ and $c_1(u) \geq 1$.
By Corollary \ref{add a flag or two}.II.B there is a desirable coloring $V_1' \cup V_2'$ of the weighted graph $(H, W_1', W_2')$ where $W_j' = W_j$ except $W_2'(x) = W_2(x)+1$ and $W_2'(y) = W_2(y)+1$.
This desirable coloring can be extended to a desirable coloring of $G$ by adding $u$ to $V_1'$ if $\{x,y\} \subseteq V_2'$ and adding $u$ to $V_2'$ otherwise.
This contradicts the criticality of $G$.

\textbf{Case 2:} $c_1(u) \geq 2$ and $c_2(u) \geq 1$.
As we are not in case 1, we may assume that $c_2(u) = 1$, and therefore by Fact \ref{basic arithmetic with potential}.2.d $\rho(\{u\}) \leq 1-\alpha+\beta$.
As $\alpha + \beta \leq 1$ we have $\rho(\{u\}) \leq 2 - 2\alpha$.
By Corollary \ref{add a flag or two}.II.C there is a desirable coloring $V_1'' \cup V_2''$ of the weighted graph $(H, W_1'', W_2'')$ where $W_j'' = W_j$ except $W_1''(x) = W_1(x)+1$ and $W_1''(y) = W_1(y)+1$.
This desirable coloring can be extended to a desirable coloring of $G$ by adding $u$ to $V_2''$ if $\{x,y\} \subseteq V_1''$ and adding $u$ to $V_1''$ otherwise.
This contradicts the criticality of $G$.

%
\end{proof}

It is well-known that potential functions are submodular.
To see this, let $A,B \subseteq V(G)$.
Let $e(A\setminus B, B \setminus A)$ denote the number of edges with an endpoint in $A \setminus B$ and another endpoint in $B \setminus A$.
By counting vertices and edges we have that 
$$ \rho(A \cup B) + \rho(A \cap B) = \rho(A) + \rho(B) - e(A \setminus B, B \setminus A) \leq \rho(A) + \rho(B).$$
In Lemma \ref{degree 3 not incident with parallel edge} we will use that our potential function satisfies the stronger property of subadditivity on proper weighted subgraphs of $G$.
That is, if $A$ or $B$ is not all of $V(G)$, then $\rho(A \cap B) \geq 0$ (this is trivial if $A \cap B = \emptyset$, otherwise apply Lemma \ref{gap lemma}).
Therefore if $A$ or $B$ is not all of $V(G)$ we have $ \rho(A \cup B) \leq \rho(A) + \rho(B)$.

\begin{lemma}\label{degree 3 not incident with parallel edge}
If $u$ is a three-two-two, then it is not in a parallel edge.
\end{lemma}
\begin{proof}
Let $u$ be as in the statement of the lemma.
Recall that means that $d(u) = 3$ and $c_1(u), c_2(u) \geq 2$.
By way of contradiction, suppose $N^m(u) = \{v, x, x\}$.
Let $H = G - u$.

Let $(H, W_1^v, W_2^v)$ be the weighted graph constructed by setting $W_j^v = W_j$ except $W_1^v(v) = W_1(v) + 1$ and $W_2^v(v) = W_2(v) + 1$.
There is no desirable coloring of $(H, W_1^v, W_2^v)$, as such a coloring could be extended to a desirable coloring of $G$ by giving $u$ the opposite color of $x$, which contradicts the choice of $G$.
So $(H, W_1^v, W_2^v)$ contains critical subgraph $(G_v, \widehat{W_1^v}, \widehat{W_2^v})$.
Let $S_v = V(G_v)$.
As $G$ does not contain a proper weighted subgraph that is critical, $v \in S_v$.
By the minimality of $G$, we have $\rho(G_v, \widehat{W_1^v}, \widehat{W_2^v}) \leq -\beta$.
As it is a subgraph, $\rho(S_v, W_1^v, W_2^v) \leq \rho(G_v, \widehat{W_1^v}, \widehat{W_2^v})$.
So $S_v$ is a vertex subset that contains $v$ and has potential bounded by $\rho(S_v, W_1, W_2) \leq \rho(S_v, W_1^v, W_2^v) + \alpha + \beta \leq \alpha$.

Let $(H, W_1^x, W_2^x)$ be the weighted graph constructed by setting $W_j^x = W_j$ except $c_1^x(x) = 0$ and $W_1^x(v) = W_1(v) + 1$.
There is no desirable coloring of $(H, W_1^x, W_2^x)$, as such a coloring could be extended to a desirable coloring of $G$ by giving $u$ color $1$, which contradicts the choice of $G$.
So $(H, W_1^x, W_2^x)$ contains critical subgraph $(G_x, \widehat{W_1^x}, \widehat{W_2^x})$.
Let $S_x = V(G_x)$.
By the minimality of $G$, we have $\rho(G_x, \widehat{W_1^x}, \widehat{W_2^x}) \leq -\beta$.
As it is a subgraph, $\rho(S_x, W_1^x, W_2^x) \leq \rho(G_x, \widehat{W_1^x}, \widehat{W_2^x})$.
By Fact \ref{basic arithmetic with potential}.4, $\rho(\{x\}, W_1, W_2) - \rho(\{x\}, W_1^x, W_2^x) \leq 1-\alpha + \beta$.

As $G$ does not contain a critical proper weighted subgraph, $S_x$ contains $x$ or $v$.
Moreover, by Corollary \ref{add a flag or two}.I, $S_x$ contains $x$.
If $S_x$ also contains $v$, then 
\begin{eqnarray*}
\rho(G[S_x] + u, W_1, W_2) &\leq & \rho(G_x, W_1^x, W_2^x) + \rho(\{x,v\}, W_1, W_2) - \rho(\{x,v\}, W_1^x, W_2^x) + \rho(\{u\}) - 3\\
		& \leq & -\beta + (1-\alpha + \beta) + (\alpha) + (2-\alpha) - 3 \\
		& \leq & -\alpha < -\beta,
\end{eqnarray*}
which contradicts Corollary \ref{at least -beta}.
So $v$ is not in $S_x$ and we have 
$$\rho(S_x, W_1, W_2) \leq \rho(V(G_x), W_1^x, W_2^x) + 1-\alpha+\beta \leq 1-\alpha. $$

Since $u$ is not in $S_x$ or $S_v$, we apply subadditivity to say that $\rho(S_x \cup S_v) \leq \alpha + (1-\alpha) = 1$.
Therefore we have 
$$\rho(G[S_x \cup S_v] + u, W_1, W_2) \leq 1 + (2-\alpha) - 3 < -\beta,$$ 
which contradicts Corollary \ref{at least -beta}.
\end{proof}

\begin{lemma}\label{triple-three next to a double constrained}
If $d_1 \geq 1$ and $u$ is a three-two-two, then there exists $v \in N(u)$ such that $v$ is doubly-constrained.
\end{lemma}
\begin{proof}
Let $u$ be as in the statement of the lemma.
Recall that means that $d(u) = 3$ and $c_1(u), c_2(u) \geq 2$.
By Lemma \ref{degree 3 not incident with parallel edge} the neighbors of $u$ are distinct.
By way of contradiction, suppose that $N(u) = \{x_1, x_2, x_3\}$ and for each $i$ there exists a color $j_i$ such that $x_i$ is $j_i$-null or $j_i$-slack.
Let $H = G - u$.
By Corollary \ref{add a flag or two}.II.D, there is a desirable coloring $V_1' \cup V_2'$ of weighted graph $(H, W_1', W_2')$ where $W_j' = W_j$ except $W_{3-j_i}'(x_i) = W_{3-j_i}(x_i) + 1$ for $i \in \{1,2,3\}$.

Consider extending the desirable coloring of $(H, W_1', W_2')$ to a coloring of $G$ by giving $u$ the color that appears the least on its neighbors, which we will call the ``first attempt coloring.''
We will show that if the first attempt coloring is not a desirable coloring, then $(H, W_1', W_2')$ satisfies certain properties, and there exists a desirable coloring of $G$ that we will call the ``flipped coloring.''

As $d(u) = 3$, in the first attempt coloring $u$ is in at most one monochromatic edge and therefore not in a monochromatic cycle.
Thus, the only way that the first attempt coloring is not a desirable coloring is if for some $i,j$ we have $x_i \in V_j$ and $|N^m[x_i] \cap V_j| > c_j(x_i)$.
Because $(H, W_1', W_2')$ is a desirable coloring, by construction we have $j = j_i$ and $ux_i$ is monochromatic.
Moreover, the desirable coloring $(H, W_1', W_2')$ with $x_i \in V_j$ implies that $x_i$ is not $j$-null.
By the assumptions of the case, this implies that $x_i$ is $j$-slack.

As $u$ is in at most one monochromatic edge, this implies that if such an $(x_i,j)$ exist, then it is unique.
By definition of $j$-slack and $|N^m[x_i] \cap V_j| > c_j(x_i)$, we have that $c_{3-j}(x_i) \geq 1$ and $N^m[x_i] \subseteq V_j$.
We construct the flipped coloring from the first attempt coloring by removing $x_i$ from $V_j$ and adding it to $V_{3-j}$.
By the above, $x_i$ is no longer in any monochromatic edges.
As $i,j$ is the unique obstruction to the first attempt coloring being a desirable coloring, we conclude that the flipped coloring is a desirable coloring of $G$.
This contradicts the choice of $G$.
\end{proof}

\begin{lemma}\label{triple three next to two somehow-constrained}
If $u$ is a triple-three with $N(u) = \{x_1, x_2, x_3\}$, then there exists $\{s,t\} \subset \{1,2,3\}$ such that $x_s$ and $x_t$ are somehow-constrained.
\end{lemma}
\begin{proof}
The outline of this argument follows similarly to the proof of Lemma \ref{degree 3 not incident with parallel edge}.
However, the details are different and sometimes more complicated, including repeated use of flipped colorings.

Let $u$ be as in the statement of the lemma.
Recall that means that $d(u) = 3$ and $c_1(u), c_2(u) \geq 3$.
By Lemma \ref{degree 3 not incident with parallel edge} the neighbors of $u$ are distinct.
Without loss of generality, by Lemma \ref{triple-three next to a double constrained} we may assume that $x_1$ is doubly-constrained.
By way of contradiction, let us assume that $x_2$ and $x_3$ are not somehow-constrained.
Let $H = G - u$.

Let $(G_1, W_1^{(1)}, W_2^{(1)})$ be the weighted graph formed from $H$ by setting $W_j^{(1)} = W_j$, except $W_j^{(1)}(x_1) = W_j(x_1) + 1$ for each $j$.
We claim that $G_1$ contains a critical subgraph.
By way of contradiction, suppose $G_1$ has a desirable coloring $V_1^{(1)} \cup V_2^{(1)}$.
We extend the desirable coloring of $G_1$ to a ``first attempt'' coloring of $G$ by giving $u$ the color that appears the least on its neighbors.
If the first attempt coloring is not a desirable coloring, then there exists an $x_i \in V_j^{(1)}$ such that $|N^m[x_i] \cap V_j^{(1)}| > c_j(x_i)$.
Because it is a desirable coloring of $G_1$, we must have that $ux_i$ is the unique monochromatic edge containing $u$ and $i \in \{2,3\}$.
So $x_i$ is not $j$-null and by assumption $x_i$ is not $j$-constrained, so $N^m[x_i] \subseteq V_j^{(1)}$, $c_j(x_i) = d(x_i)$, and therefore $c_{3-j}(x_i) > 0$.
We construct the ``flipped coloring'' from the first attempt coloring by removing $x_i$ from $V_j^{(1)}$ and adding it to $V_{3-j}^{(1)}$.
The first attempt coloring or the flipped coloring is a desirable coloring of $G$, which contradicts the choice of $G$.
This proves the claim.

Let $(\widehat{G_1}, \widehat{W_1^{(1)}}, \widehat{W_2^{(1)}})$ be a critical weighted subgraph of $G_1$.
Let $S_1 = V(\widehat{G_1})$.
By the minimality of $G$ we have that $\rho(\widehat{G_1}) \leq -\beta$ and $x_1 \in S_1$.
Therefore there exists a vertex set $S_1$ that contains $x_1$ and $\rho(S_1, W_1, W_2) \leq -\beta + \alpha + \beta = \alpha$.

Let $(G_{2,3}, W_1, W_2)$ be the weighted graph formed from $H$ by adding edge $x_2x_3$.
We claim that $G_{2,3}$ contains a critical subgraph.
By way of contradiction, suppose $G_{2,3}$ has a desirable coloring $V_1^{(2,3)} \cup V_2^{(2,3)}$.
We extend the desirable coloring of $G_{2,3}$ to a ``first attempt'' coloring of $G$ by giving $u$ the opposite color of the one on $x_1$.
The first attempt coloring does not contain a monochromatic cycle, as any such cycle containing $u$ would be a monochromatic cycle in $G_{2,3}$ using edge $x_2x_3$.
So, if the first attempt coloring is not a desirable coloring, then there exists an $x_i \in V_j^{(2,3)}$ such that $|N^m[x_i] \cap V_j^{(2,3)}| > c_j(x_i)$.
Because it is a desirable coloring of $G_{2,3}$, $x_i$ must be in strictly more monochromatic edges in $G$: so $ux_i$ is a monochromatic edge and $x_2x_3$ is not a monochromatic edge.
This implies that $ux_{5-i}$ is not a monochromatic edge, and therefore $i$ is unique.
We create the ``flipped coloring'' from the first attempt coloring by removing $x_i$ from $V_j^{(2,3)}$ and adding it to $V_{3-j}^{(2,3)}$.
The first attempt coloring or the flipped coloring is a desirable coloring of $G$, which contradicts the choice of $G$.
This proves the claim.

\begin{sloppypar}
Let $(\widehat{G_{2,3}}, \widehat{W_1^{(2,3)}}, \widehat{W_2^{(2,3)}})$ be a critical weighted subgraph of $G_{2,3}$.
Let $S_{2,3} = V(\widehat{G_{2,3}})$.
By the minimality of $G$ we have that $\rho(\widehat{G_{2,3}}) \leq -\beta$ and $\{x_2,x_3\} \subseteq S_{2,3}$.
Therefore there exists a vertex set $S_{2,3}$ that contains $\{x_2, x_3\}$ and $\rho(S_{2,3}, W_1, W_2) \leq -\beta + 1$.
\end{sloppypar}

Since $u$ is not in $S_1$ or $S_{2,3}$, we apply subadditivity to say that $\rho(S_1 \cup S_{2,3}) \leq \alpha + (1-\beta)$.
Therefore we have 
$$\rho(G[S_1 \cup S_{2,3}] + u, W_1, W_2) \leq (\alpha + 1 - \beta) + (2-\alpha) - 3 = -\beta,$$ 
which contradicts Corollary \ref{at least -beta}.
\end{proof}

\section{Discharging} \label{discharging section}

To each vertex $u$ we assign a \emph{charge}, $ch(u) = \rho(\{u\}) - d(u)/2$.
By construction we have $\rho(G) = \sum_u ch(u)$.
Our goal is to show $\sum_u ch(u) \leq -\beta$, which will prove the theorem by contradicting the choice of $G$.

Recall $\alpha = \frac{d_2 + 2}{(d_1 + 2)(d_2 + 1)}$, $\beta = \frac{1}{d_2 + 1}$, and $d_2 \geq 2d_1 + 2$.
Therefore $\frac{1}{d_1 + 2} < \alpha \leq \frac{2}{2d_1+3}$ and $\beta \leq \frac{1}{2d_1+3}$.
So we have 
\begin{itemize}
	\item For $d_1 = 0$ we have $\alpha = 0.5 + \beta/2$ and $\beta \leq 1/3$. 
	\item For $d_1 = 1$ we have $\alpha \leq \frac{2}{5}$ and $\beta \leq \frac{1}{5}$.
	\item For $d_1 = 2$ we have $\alpha \leq \frac{2}{7}$ and $\beta \leq \frac{1}{7}$.
	\item For $d_1 = 3$ we have $\alpha \leq \frac{2}{9}$ and $\beta \leq \frac{1}{9}$.
	\item For $d_1 \geq 4$ we have $\alpha \leq \frac{2}{11}$  and $\beta \leq \frac{1}{11}$.
\end{itemize}
We define functions $D, S$ over the whole numbers as follows:
\begin{itemize}
	\item $D(0) = S(0) = 0$.
	\item $D(1) = D(2) = 0.5-\alpha$ and $S(1) = S(2) = 0$.
	\item $D(3) = 0.5 - \alpha - \beta/5$ and $S(3) = \beta/5$.
	\item $D(d) = 0.5 - \alpha - \beta$ and $S(d) = \beta$ for $d \geq 4$.
\end{itemize}
By construction, for each $k$ we have $D(k) + S(k) \geq 0.5 - \alpha$ and $D(k) \geq S(k)$.
Moreover, $D(1) \geq S(1) + \beta/2$ and for $k \geq 2$ we have $D(k) \geq S(k) + \beta$.

We create a ``final charge'' function $ch^*$. 
The function $ch^*$ is the outcome of modifying $ch$ according to two discharging rules.
In a discharging rule, the act of ``vertex $u$ gives charge $x$ to neighbor $w$'' describes increasing $ch(u)$ by $x$ and decreasing $ch(w)$ by $x$.
In this way we have the property that $\sum_{u \in V(G)}ch(u) = \sum_{w \in V(G)} ch^*(w)$.
The discharging rules are:
\begin{itemize}
	\item[R1] Each vertex that is doubly-constrained gives $D(d_1)$ charge to each neighbor that is three-two-two and not doubly-constrained.
	\item[R2] Each vertex that is somehow-constrained but not doubly-constrained gives $S(d_1)$ charge to each neighbor that is triple-three.
\end{itemize}
As $0 = D(0) = S(0) = S(1) = S(2)$, we effectively only apply R1 when $d_1 \geq 1$ and apply R2 when $d_1 \geq 3$.

Our goal is to show that for every vertex $u$ we have $ch^*(u) \leq 0$.
Moreover, there exists two vertices $y_1, y_2$ such that $ch^*(y_j) \leq -\beta/2$ for each $j$, or there exists a vertex $y$ such that $ch^*(y) \leq -\beta$
This will prove $\rho(G) = \sum_{w \in V(G)} ch^*(w) \leq -\beta$, which contradicts the choice of $G$ and proves the theorem.

\begin{lemma}\label{get d1=0 out of the way first}
If $d_1 = 0$, then for any vertex $u$ we have $ch(u) \leq - \beta/2$.
\end{lemma}
\begin{proof}
Let $u$ be a vertex.
By Lemma \ref{2 neighbors}, $d(u) \geq 2$.
If $d(u) \geq 3$, then $ch(u) \leq 2 - \alpha - 3/2 = - \beta/2$.
By Lemma \ref{degree 2}, if $d(u) = 2$, then one of the following three cases applies.
\begin{itemize}
	\item $c_1(u) = c_2(u) = 1$, which implies $ch(u) = \alpha-1 = (\beta-1)/2$. Because $\beta \leq 1/3$, this implies $ch(u) \leq -\beta$.
	\item $c_1(u) = 0$, which by Fact \ref{basic arithmetic with potential}.2.c implies $ch(u) \leq -\beta$.
	\item $c_2(u) = 0$, which by Fact \ref{basic arithmetic with potential}.2.d implies $ch(u) \leq -\beta$.
\end{itemize}
\end{proof}

\begin{lemma}
If $u$ is doubly constrained, then 
\begin{itemize}
	\item if $d_1 = 1$, then $ch(u) \leq -d(u)(0.5 - \alpha) - \beta$,
	\item if $d_1 = 2$, then $ch(u) \leq -d(u)(0.5 - \alpha)$,
	\item if $d_1 = 3$ and $d(u) \neq 5$, then $ch(u) \leq -d(u)(0.5 - \alpha)$,
	\item if $d_1 = 3$ and $d(u) = 5$, then $ch(u) \leq -d(u)(0.5 - \alpha) + \beta$, and
	\item if $d_1 \geq 4$, then $ch(u) \leq -d(u)(0.5 - \alpha - \beta) - \beta$.
\end{itemize}
\end{lemma}
\begin{proof}
Because $u$ is doubly constrained, we have that $c_1(u),c_2(u) \leq d(u) -  1$.
Therefore 
$$\rho(\{u\}) = \alpha c_1(x)  +  \beta (c_2(x)-1) \leq (d(u)-1)(\alpha + \beta) - \beta, $$
and so 
$$ ch(u) \leq -d(u)(0.5 - \alpha - \beta) - 2\beta - \alpha. $$
This proves the lemma when $d_1 \geq 4$.

Recall that $c_1(u) \leq d_1 + 1$.
So we can get a refined estimate as 
$$ \rho(\{u\}) \leq \min\{d_1+1, d(u)-1\}\alpha + (d(u)-2)\beta. $$
This implies that
$$ ch(u) \leq -d(u)(0.5 - \alpha) + (\min\{d_1+1, d(u)-1\}-d(u))\alpha + (d(u)-2)\beta. $$
As $\alpha \geq 2\beta$ we can simplify this inequality to 
$$ ch(u) \leq -d(u)(0.5 - \alpha) + (2\min\{d_1+1, d(u)-1\}-d(u)-2)\beta. $$

Observe that $2\min\{d_1+1, d(u)-1\}-d(u)-2 = d_1 - 2 - |d_1 + 2 - d(u)|$. 
Plugging that value in gives us 
$$ ch(u) \leq -d(u)(0.5 - \alpha) + (d_1 - 2)\beta, $$
with an additional $\beta$ subtracted from the right hand side when $d_1 = 3$ and $d(u) \neq 5$.
This proves the lemma.
\end{proof}

\begin{cor}\label{doubly constrained charge}
If $u$ is doubly constrained, then $ch^*(u) \leq 0$.
Moreover, if $d_1 = 1$, then $ch^*(u) \leq -\beta$.
\end{cor}

\begin{lemma}\label{potential of somehow constrained}
If $u$ is somehow constrained and not three-two-two, then:
\begin{itemize}
	\item If $d_1 \in \{1,2\}$, then $ch(u) \leq -\beta$.
	\item If $d_1 = 3$, then $ch(u) \leq -d(u)\beta/5 - \beta$.
	\item If $d_1 \geq 4$, then $ch(u) \leq -(d(u)+1)\beta$.
\end{itemize}
\end{lemma}
\begin{proof}
Recall that every vertex has potential at most $2-\alpha$, so
$$ch(u) \leq (2-\alpha)-d(u)/2 \leq -(d(u)-4)/2 - \alpha.$$
Observe that at $d(u) = 5$ with $\alpha \geq 2\beta$ this becomes $-1/2-2\beta$, which is low enough to prove the lemma.
Since the slope---in respect to the degree of $u$---of the right hand side is $-0.5$, which is far smaller than $-\beta$, this proves the lemma for all vertices with degree at least $5$.

If $d(u) = 2$, then by Lemma \ref{degree 2} one of the following three cases applies.
\begin{itemize}
	\item $c_1(u) = c_2(u) = 1$, which implies $ch(u) = \alpha-1$.
	\item $c_1(u) = 0$ and $c_2(u) = 2$, which implies $ch(u) = \beta - 1$.
	\item $c_2(u) = 0$ and $c_1(u) = 2$, which implies $ch(u) = 2\alpha - \beta - 1$.
\end{itemize}
Therefore $ch(u) \leq 2\alpha-1 - \beta$.
Observe that 
\begin{itemize}
	\item For $d_1 \in \{1,2\}$ we have $\alpha \leq \frac{2}{5}$, so $2\alpha-1 < 0$.
	\item For $d_1 = 3$ we have $\alpha \leq \frac{2}{9}$ and $\beta \leq \frac{1}{9}$, so $2\alpha-1 \leq -2\beta/5$.
	\item For $d_1 \geq 4$ we have $\alpha \leq \frac{2}{11}$  and $\beta \leq \frac{1}{11}$, so $2\alpha-1 \leq -2\beta$.
\end{itemize}
Thus the lemma is proven when $d(u) = 2$.

If $d(u) = 3$ and $u$ is not three-two-two, then $c_i(u) \leq 1$ for some $i$.
If $i=1$, then $\rho(\{u\}) \leq 1 + \alpha - \beta$, and if $i=2$, then $\rho(\{u\}) \leq 1 - \alpha + \beta$.
Therefore $ch(u) \leq -0.5 + \alpha - \beta.$
\begin{itemize}
	\item If $d_1 \in \{1, 2\}$, then the lemma immediately follows from $\alpha \leq 0.5$.
	\item If $d_1 = 3$, then $ch(u) \leq  -8\beta/5$.
	\item If $d_1 \geq 4$, then $ch(u) \leq \frac{-7}{22} - \beta < -4\beta$.
\end{itemize}
Thus the lemma is proven when $d(u) = 3$.

We now consider the final case where $d(u) = 4$.
As $ch(u) \leq -(d(u)-4)/2 - \alpha = -\alpha \leq -2\beta$, the lemma is proven for $d_1 \leq 3$.
So suppose $d_1 \geq 4$.
As $u$ is somehow constrained, we have $c_i(u) \leq 3$ for some $i$.
If $i=1$, then $\rho(\{u\}) \leq 1 + 3\alpha - \beta$, and if $i=2$, then $\rho(\{u\}) \leq 1 - \alpha + 3\beta$.
Therefore $ch(u) \leq 3\alpha - \beta - 1$.
As $\alpha \leq \frac{2}{11}$ and $\beta \leq \frac{1}{11}$, we have that $ch(u) < -5\beta$.
\end{proof}

\begin{cor}\label{singly constrained charge}
If $d_1 \geq 1$ and $u$ is somehow-constrained, not three-two-two, and not doubly constrained, then $ch^*(u) \leq -\beta$.
\end{cor}

\begin{lemma}\label{charge for everything else}
If $u$ is not three-two-two or somehow-constrained, then $ch^*(u) \leq -\beta$.
\end{lemma}
\begin{proof}
If $u$ is not three-two-two or somehow-constrained, then $u$ is not involved in either discharging rule and so $ch(u) = ch^*(u)$.

If $u$ is not three-two-two, not somehow-constrained, and $d(u) = 3$, then $u$ is $i$-null for some $i$.
If $d(u) = 2$ and not somehow-constrained, then by Lemma \ref{degree 2} $u$ is $i$-null for some $i$.
If $u$ is $i$-null and $d(u) \geq 2$, then by Fact \ref{basic arithmetic with potential}.2 we have $ch(u) \leq -\beta$.
So assume $d(u) \geq 4$, which implies $ch(u) \leq (2-\alpha)-2 \leq -\beta$.
\end{proof}

\begin{lemma}\label{discharge three-two-two}
If $d_1 \geq 1$, $u$ is three-two-two, not triple-three, and not doubly-constrained, then $ch^*(u) \leq 0$.
Moreover, if $d_1 \geq 2$, then $ch^*(u) \leq -\beta$.
\end{lemma}
\begin{proof}
If $u$ is three-two-two but not triple-three, then $d(u) = 3$ and there exists an $i$ such that $c_i(u) = 2$.
This implies that $u$ is somehow-constrained.
By Lemma \ref{triple-three next to a double constrained}, $u$ has a neighbor that is doubly-constrained (which implies that it is not triple-three).
So while applying the discharging rules, at least one neighbor will give $D(d_1)$ charge to $u$ while $u$ will give at most $2S(d_1)$ charge to its neighbors, and therefore
$$ ch^*(u) \leq \rho(\{u\}) - 3/2 - D(d_1) + 2S(d_1) . $$

If $d_1 \in \{1,2\}$, then $D(d_1) = 0.5-\alpha$ and $S(d_1) = 0$, so $ch^*(u) \leq (2-\alpha)-3/2-(0.5-\alpha)=0$.
This proves the lemma if $d_1 = 1$, so assume $d_1 \geq 2$.

Because $c_i(u) = 2$ we have that $W_i(u) = d_i - 1$, which is positive.
So we have that $\rho(\{u\}) \leq (2-\alpha) - \min\{\alpha(d_1-1), \beta(d_2 - 1)\}$.
As $\alpha \geq 2\beta$ and $d_2 \geq 2d_1 + 2$, this bound becomes $\rho(\{u\}) \leq (2-\alpha) - 2(d_1-1)\beta$.
Plugging this into the above inequality creates 
$$ ch^*(u) \leq 1/2 -\alpha - 2(d_1-1)\beta - D(d_1) + 2S(d_1) . $$
When $d_1 = 2$ this simplifies to $ch^*(u) \leq -2\beta$, so assume $d_1 \geq 3$. 

Recall that $D(3) = 0.5 - \alpha - \beta/5$ and $S(3) = \beta/5$.
So when $d_1 = 3$ we have that 
\begin{eqnarray*}
	ch^*(u)	
		& 	\leq	& 0.5 - \alpha - 4\beta - (0.5 - \alpha - \beta/5) + 2\beta/5 \\ 
		& 	= 	& -3.4\beta .
\end{eqnarray*}

Recall that $D(d_1) = 0.5 - \alpha - \beta$ and $S(d_1) = \beta$ for $d_1 \geq 4$.
So when $d_1 \geq 4$ we have that 
\begin{eqnarray*}
	ch^*(u)	
		&	\leq	& 0.5 - \alpha - 6\beta - (0.5 - \alpha - \beta) + 2\beta \\ 
		& 	= 	& -3\beta .
\end{eqnarray*}
\end{proof}

\begin{lemma}\label{discharge triple three}
If $d_1 > 0$ and $u$ is triple-three, then $ch^*(u) \leq 0$.
Moreover, if $u$ is adjacent to two doubly-constrained vertices and $d_1 \geq 2$, then $ch^*(u) \leq -\beta$.
\end{lemma}
\begin{proof}
A triple-three vertex is not somehow-constrained.
By Lemmas \ref{triple-three next to a double constrained} and \ref{triple three next to two somehow-constrained}, there exists $\{w_1, w_2\} \subset N(u)$ such that $w_1$ is doubly-constrained and $w_2$ is somehow-constrained.

Suppose first that $w_2$ is not doubly-constrained.
While applying the discharging rules, $w_1$ gives $D(d_1)$ charge to $u$ and $w_2$ gives $S(d_1)$ charge to $u$.
Recall that by construction, $S(d_1) + D(d_1) \geq 0.5-\alpha$.
Therefore $ch^*(u) \leq (2-\alpha) - 1.5 - (0.5-\alpha) = 0$.

Now suppose that $w_2$ is doubly-constrained; this will prove the ``moreover'' part of the lemma.
Recall that for $k \geq 2$ we have $D(k) \geq S(k) + \beta$.
During discharging, $u$ is given charge $2D(d_1)$ instead of $D(d_1) + S(d_1)$ as considered before, which proves the lemma.
\end{proof}

We are now ready to finish the proof to Theorem \ref{central goal}.
The proof is immediate by Lemma \ref{get d1=0 out of the way first} if $d_1 = 0$ as $G$ contains at least two vertices, so assume $d_1 \geq 1$ for the rest of the section.
Recall that triple-three is stronger than three-two-two and doubly-constrained is stronger than somehow-constrained.
Every vertex $u$ of $G$ falls into one of these categories:
\begin{enumerate}
	\item $u$ is triple-three, 
	\item $u$ is doubly-constrained,
	\item $u$ is three-two-two, but not triple-three or doubly-constrained, 
	\item $u$ is somehow-constrained, but not three-two-two or doubly-constrained, or
	\item $u$ is not somehow-constrained or three-two-two.
\end{enumerate}
By Corollaries \ref{doubly constrained charge} and \ref{singly constrained charge} and Lemmas \ref{charge for everything else}, \ref{discharge three-two-two}, and \ref{discharge triple three}, $ch^*(u) \leq 0$ for any vertex $u$.
Moreover, by Corollary \ref{singly constrained charge} and Lemma \ref{charge for everything else}, $ch^*(u) \leq -\beta$ if $u$ is not doubly-constrained or three-two-two.
As $-\beta < \rho(G) = \sum_u ch^*(u)$ and $G$ contains at least two vertices, this implies that each vertex $u$ of $G$ is category (1), (2), or (3).

If every vertex is category (1), (2), or (3), then by Lemma \ref{triple-three next to a double constrained} there exists $w \in V(G)$ such that $w$ is doubly-constrained.
By Corollary \ref{doubly constrained charge} $ch^*(w) \leq -\beta$ if $d_1 = 1$, which is a contradiction, and so $d_1 \geq 2$.

By Lemma \ref{discharge three-two-two} and $d_1 \geq 2$, if $u$ is category (3), then $ch^*(u) \leq -\beta$.
So all vertices in $u$ are doubly-constrained or triple-three.
Vertices that are triple-three are not somehow-constrained.
So if $u$ is triple-three, then the two somehow-constrained neighbors of $u$ found in Lemma \ref{triple three next to two somehow-constrained} are doubly-constrained, and so by Lemma \ref{discharge triple three} $ch^*(u) \leq -\beta$.
This contradicts the choice of $G$ as $-\beta < \rho(G) = \sum_u ch^*(u) \leq -\beta$, and it proves Theorem \ref{central goal}.

\section{Algorithm} \label{algorithm section}

Our proof to Theorem \ref{central goal} is constructive, although it is not immediately clear that it leads to a polynomial-time algorithm.
Central to the algorithm is a routine that will find a nontrivial proper subgraph that minimizes potential (we restrict to nonempty subgraphs because the gap lemma requires finding subgraphs $H$ with $0 < \rho(H) < \alpha - \beta$, and the empty subgraph has potential $0$).
For this routine we apply Corollary 2.4 of \cite{CranstonYancey2020} with $m_1 = m_2 = 1$.
In some situations we merely want to confirm whether the assumptions of Theorem \ref{central goal} hold (in other words, whether $\min_H \rho(H) > -\beta$), which can be done using $m_1 = m_2 = 0$.

\begin{theorem}[\cite{CranstonYancey2020}]\label{gap lemma algorithm}
Let $G$ be a weighted graph as described in Section \ref{notation subsection}.
We can find a weighted subgraph that minimizes potential among those that are  nonempty and non-spanning in $O(n^4 \log(n))$ time.
We can find a weighted subgraph of $G$ that minimizes potential in $O(n^2 \log(n))$ time.
\end{theorem}

We will now prove the following theorem.

\begin{theorem}\label{algorithm theorem}
Let $G$ be a weighted graph as described in Section \ref{notation subsection}.
If $\rho(H) > -\beta$ for every weighted subgraph $H$ of $G$, then a desirable coloring of $G$ can be found in $O(n^5 \log(n))$ time.
\end{theorem}
\begin{proof}
Our algorithm breaks into a series of cases.
The cases directly follow the arguments of Section \ref{reducible configurations section}, which we refer to frequently.
We can simplify some of the details for the algorithm, such as a smaller set of assumptions for using the ``flipped coloring'' in cases 6 and 7 below.

The argument presented in Section \ref{discharging section} implies that if $G$ satisfies the assumptions of Theorem \ref{gap lemma algorithm}, then at least one of the cases will apply.
The later cases assume that the earlier cases did not apply.
Our calculation of computational cost includes determining if a given case applies, and the cost to perform the ensuing operations if it does apply.

Let $n = |V(G)|$.
As $\rho(G) > -\beta$ we have that $|E(G)| \leq O(n)$.
We will construct $T: \mathbb{N} \rightarrow \mathbb{N}$ such that a desirable coloring of $G$ can be found in $T(n)$ computations.
We will show that $T(n_*) \leq C n_*^4 \log(n_*) + \max_{n' < n_*}T(n')$ for some fixed constant $C$ and arbitrary $n_*$, which implies the theorem.

\textbf{Case 1:} $G$ is disconnected, an edge has multiplicity at least $3$, $c_i(u) < 0$ for some vertex $u$, or $n \leq 3$.
(This case parallels Section \ref{opening statements section}.)
Removing edges until each has multiplicity at most $2$ can be done in linear time.
Decreasing the weight of each vertex until the capacities are nonnegative can be done in linear time.
If $n \leq O(1)$, then the desirable coloring can be done in $O(1)$ time by trying all $2^n$ possibilities.

Finding the connected components can be done in linear time, and recursing on each connected component can be done in time $O(n) + \max_{k_1 + k_2 + \cdots k_\ell = n} \sum_i T(k_i)$.
Because $T$ is a convex function, this is maximized when $\ell = 1$, which is when $G$ is connected.

\textbf{Case 2:} $G$ contains a nonempty proper weighted subgraph $H$ such that $\rho(H) \leq \alpha - \beta$.
(This case parallels Lemma \ref{gap lemma}.)
As in the proof to Lemma \ref{gap lemma}, such an $H$ will be non-spanning.
So we use Theorem \ref{gap lemma algorithm} to find the $H$ that minimizes potential.
This is done in $O(n^4 \log(n))$ time.
Let $n_H = |V(H)|$.

If $\rho(H) > 0$, then we increase $W_2$ on a vertex in $H$ with a neighbor outside of $H$.
By the proof to Lemma \ref{gap lemma}, $H$ satisfies the assumptions of Theorem \ref{algorithm theorem}.
Therefore we recurse on $H$ to find a desirable coloring of $H$.
This is done in $T(n_H)$ time.

We construct $(G^*, W_1^*, W_2^*)$ as in the proof to Lemma \ref{gap lemma}.
By the proof to Lemma \ref{gap lemma}, we have that $G^*$ satisfies assumptions of Theorem \ref{algorithm theorem}.
Therefore we recurse on $G^*$ to find a desirable coloring of $G^*$.
The coloring of $H$ and the coloring of $G^*$ combine to form a desirable coloring of $G$, finishing the algorithm.
This is done in $O(n) + T(n-n_H)$ time.

The overall running time of this case is then $O(n^4 \log(n)) + \max_{n_H} T(n_H) + T(n-n_H)$.
As $T$ is convex, the maximization occurs when $n_H \leq O(1)$ or $n-n_H \leq O(1)$.
As $H$ is a non-spanning and nonempty, we have that $n_H < n$ and $n-n_H < n$.
Therefore the overall running time of the algorithm in this case is $O(n^4 \log(n)) + T(n-1)$.

\textbf{Case 3:} there exists a vertex $u$ with $|N(u)| = 1$.
(This case parallels Lemma \ref{2 neighbors}.)
Such a vertex can be found in $O(n)$ time.
We recurse on $G$ after removing $u$ and modifying the weights on the neighbor of $u$ if $u$ is in a parallel edge or $i$-null for some $i$.
By the proof to Lemma \ref{2 neighbors}, the smaller graph satisfies the assumptions of Theorem \ref{algorithm theorem}, and we can find a desirable coloring of it by recursing.
The coloring that is returned from the recursive call is easily extended to $G$.
The running time of the algorithm in this case is $O(n) + T(n-1)$.

\textbf{Case 4:} there exists a $u$ with $d(u) = 2$, $\min(c_1(u), c_2(u)) \geq 1$, and $\max(c_1(u), c_2(u)) \geq 2$.
(This case parallels Lemma \ref{degree 2}.)
Such a vertex can be found in $O(n)$ time.
We recurse on $G$ after removing $u$ and modifying the weights on the neighbors of $u$ as described in the proof to Lemma \ref{degree 2}.
By the proof to Lemma \ref{degree 2}, the smaller graph satisfies the assumptions of Theorem \ref{algorithm theorem}, and we can find a desirable coloring of it by recursing.
The coloring that is returned from the recursive call is easily extended to $G$.
The running time of the algorithm in this case is $O(n) + T(n-1)$.

\textbf{Case 5:} there exists a three-two-two $u$ that is in a parallel edge.
(This case parallels Lemma \ref{degree 3 not incident with parallel edge}.)
Let $N^m(u) =  \{v, x, x\}$.
Such a vertex can be found in $O(n)$ time.
As in the proof to Lemma \ref{degree 3 not incident with parallel edge}, we construct weighted graphs $(H, W_1^v, W_2^v)$ and $(H, W_1^x, W_2^x)$ in $O(n)$ time.
By the proof to Lemma \ref{degree 3 not incident with parallel edge}, at least one of $(H, W_1^v, W_2^v)$ or $(H, W_1^x, W_2^x)$ satisfies the assumptions of Theorem \ref{algorithm theorem}.
We use Theorem \ref{gap lemma algorithm} to determine in $O(n^2 \log(n))$ time which of the weighted graphs satisfies the assumptions of Theorem \ref{algorithm theorem}.
We recurse on one of the weighted graphs that satisfies the assumptions of Theorem \ref{algorithm theorem} (and \emph{only} one, even if they both satisfy the assumptions), which will return a coloring that is easily extended to a desirable coloring of $G$.
The running time of the algorithm in this case is $O(n^2 \log(n)) + T(n-1)$.

\textbf{Case 6:} there exists a three-two-two $u$ that has no doubly-constrained neighbor.
(This case parallels Lemma \ref{triple-three next to a double constrained}.)
Such a vertex can be found in $O(n)$ time.
We construct and recurse to find a desirable coloring $V_1' \cup V_2'$ of $(H, W_1', W_2')$ as in the proof to Lemma \ref{triple-three next to a double constrained}.
Let $N(u) = \{x_1, x_2, x_3\}$, and let $j$ be such that $|V_j' \cap N(u)| \geq 2$.
If the following conditions hold for some $i \in \{1, 2, 3\}$
\begin{itemize}
	\item $x_i \notin V_j'$,
	\item $c_j(x_i) > 0$, 
	\item $N(x_i) \cap V_j' = \emptyset$, 
\end{itemize}
then we remove $x_i$ from $V_{3-j}'$ and add it to $V_{j}'$.
This adjustment can be done in $O(n)$ time (verifying whether the conditions hold takes $d(x_i)$ time, and $d(x_i) \leq O(n)$).
By the proof to Lemma \ref{triple-three next to a double constrained}, the coloring then extends to a desirable coloring of $G$ by adding $u$ to $V_{3-j}'$.
The running time of the algorithm in this case is $O(n) + T(n-1)$.

\textbf{Case 7:} there exists a triple-three $u$ with two neighbors that are not somehow-constrained.
(This case parallels Lemma \ref{triple three next to two somehow-constrained}.)
Such a vertex can be found in $O(n)$ time.
As in the proof to Lemma \ref{triple three next to two somehow-constrained}, we construct weighted graphs $(G_1, W_1^{(1)}, W_2^{(1)})$ and $(G_{2,3}, W_1, W_2)$ in $O(n)$ time.
By the proof to Lemma \ref{triple three next to two somehow-constrained}, at least one of $(G_1, W_1^{(1)}, W_2^{(1)})$ and $(G_{2,3}, W_1, W_2)$ satisfies the assumptions of Theorem \ref{algorithm theorem}.
We use Theorem \ref{gap lemma algorithm} to determine in $O(n^2 \log(n))$ time which of the weighted graphs satisfies the assumptions of Theorem \ref{algorithm theorem}.
We recurse on one of the weighted graphs that satisfies the assumptions of Theorem \ref{algorithm theorem} to produce a desirable coloring $V_1' \cup V_2'$.
Let $N(u) = \{x_1, x_2, x_3\}$, and let $j$ be such that $|V_j' \cap N(u)| \geq 2$.
If the following conditions hold for some $i \in \{1, 2, 3\}$
\begin{itemize}
	\item $x_i \notin V_j'$,
	\item $c_j(x_i) > 0$, 
	\item $N(x_i) \cap V_j' = \emptyset$, 
\end{itemize}
then we remove $x_i$ from $V_{3-j}'$ and add it to $V_{j}'$.
This adjustment can be done in $O(n)$ time (verifying whether the conditions hold takes $d(x_i)$ time, and $d(x_i) \leq O(n)$).
By the proof to Lemma \ref{triple three next to two somehow-constrained}, the coloring then extends to a desirable coloring of $G$ by adding $u$ to $V_{3-j}'$.
The running time of the algorithm in this case is $O(n^2 \log(n)) + T(n-1)$.

\end{proof}

\section{An Open Question} \label{conjecture section}

As a graduate student, our first experience with the potential method was an attempt at improving the $d_2 \geq 2d_1 + 2$ assumption in Theorem \ref{Borodin Kostochka theorem}, which is a question we continue to think about many years later.
We use the assumption $d_2 \geq 2d_1 + 2$ in several places to prove Theorem \ref{main theorem}---but not all of those arguments are needed to prove Theorem \ref{Borodin Kostochka theorem}.
The arguments in this paper can be used to create a proof of Theorem \ref{Borodin Kostochka theorem} where the only time $d_2 \geq 2d_1 + 2$ is used is in Fact \ref{basic arithmetic with potential}.6, exclusively to to prove Fact \ref{basic arithmetic with potential}.7, exclusively to prove Corollary \ref{add a flag or two}.II.B, exclusively to prove Lemma \ref{degree 2}.

The constructions of sparse critical graphs in Section \ref{arithmetic and constructions section} do not depend on $d_2 \geq 2d_1 + 2$ (only $d_2 > d_1$), so an improvement to Theorem \ref{Borodin Kostochka theorem} would be sharp.
However, we can show that such constructions are not sharp when $d_2 = d_1 + 1$, and therefore any such improvement to Theorem \ref{Borodin Kostochka theorem} would still require at least $d_2 \geq d_1 + 2$ as an assumption.
We consider a new gadget that we name the \emph{double-pennon}.
It might be viewed as a generalization of the butterfly graph used by Borodin and Kostochka \cite{BorodinKostochka2011} for $(\Delta_0, \Delta_1)$-coloring, although it satisfies weaker properties.

A double-pennon is attached to $u$ by
\begin{enumerate}
	\item adding unweighted vertices $x_*, x_1, \ldots x_{d_2}$ and edges $x_*x_i, ux_i$ for each $i$ and $x_*u$, and then 
	\item adding unweighted vertices $y_*, y_1, \ldots y_{d_2}$ and edges $y_*y_i, x_*y_i$ for each $i$ and $y_*x_*$. 
\end{enumerate}
Adding a double-pennon adds $2(d_2 + 1)$ unweighted vertices and $4d_2 + 2$ edges.
In Theorem \ref{double-pennon} we will show that attaching a double-pennon to $u$ has the same effect as decreasing each of $c_1(u)$ and $c_2(u)$ by one.

When $d_2 = d_1 + 1$, attaching a double-pennon to $u$ can be interpreted as attaching a flag (as mentioned in Section \ref{arithmetic and constructions section}) to $u$ and then attaching a second flag to the center of the first flag.
If $d_2 = d_1 + 1$, then adding a double-pennon changes the potential by of an entire graph by $2(d_1+2)(2 - \alpha) - (4d_1 + 6) = -2\beta$.
As increasing each capacity of $u$ by one increases the potential by $\alpha + \beta > 2\beta$, this implies that the double-pennon can be used to create sparser critical graphs when $d_2 = d_1 + 1$ than those constructed in Section \ref{arithmetic and constructions section}.

\begin{theorem}\label{double-pennon}
Let $d_2 > d_1$ and $x_*, x_1, \ldots x_{d_2}, y_*, y_1, \ldots y_{d_2}$ be the vertex set of a double-pennon attached at $u$ as described above in graph $H$.
In any desirable coloring $V_1 \cup V_2$ of $H$, we have $\{x_*, x_1, \ldots, x_{d_2}\} \cap V_i \neq \emptyset$ for each $i$.
\end{theorem}
\begin{proof}
Let $V_1 \cup V_2$ be a desirable coloring of $H$.
As $H[V_1]$ has maximum degree at most $d_1 < d_2$ and $H[\{x_*, x_1, \ldots, x_{d_2}\}] \cong K_{1,d_2}$, it follows that $V_2 \cap \{x_*, x_1, \ldots, x_{d_2}\} \neq \emptyset$.
Symmetrically, there exists $\widehat{y} \in V_2 \cap \{y_*, y_1, \ldots, y_{d_2}\}$.
Because $H[\{\widehat{y}, x_*, x_1, \ldots, x_{d_2}\}] \cong K_{1,d_2+1}$ and $H[V_2]$ has maximum degree at most $d_2$, it follows that $V_1 \cap \{x_*, x_1, \ldots, x_{d_2}\} \neq \emptyset$.
\end{proof}

It is unknown if double-pennons are part of an optimal construction of sparse $(\Delta_{d_1}, \Delta_{d_1 + 1})$-critical graph; even the optimal sparsity condition for $(\Delta_1, \Delta_2)$-coloring remains open.
Kostochka, Xu, and Zhu \cite{KostochkaXuZhu2023} recently made progress on $(\Delta_1, \Delta_3)$-coloring, but this is still open as well.
We would like to know if the constructions of Section \ref{arithmetic and constructions section} are optimal when $d_1 + 2 \leq d_2 \leq 2d_1 + 1$.

\begin{ques}
Does there exist a $(\Delta_{d_1}, \Delta_{d_2})$-critical graph $G$ with $d_1 + 2 \leq d_2$ and $\rho(G) > -\beta$?
\end{ques}

\bibliographystyle{alpha}
\bibliography{sparseForests}

\end{document}